\documentclass[10pt]{amsart}
\usepackage{verbatim}
\usepackage{eucal,url,amssymb,stmaryrd,enumerate,amscd,array}
\usepackage{amsmath}
\usepackage{amsfonts}
\usepackage{amsthm}

\numberwithin{equation}{section}
\newtheorem{thrm}{Theorem}[section]
\newtheorem{lemma}[thrm]{Lemma}
\newtheorem{prop}[thrm]{Proposition}

\newtheorem{conv}[thrm]{Convention}
 1

\begin{document}

\begin{abstract}
We present a simple explicit construction of hyper-K\"ahler and
hyper-symplectic (also known as neutral hyper-K\"ahler or
hyper-parak\"ahler) metrics in 4D using the Bianchi type groups
of class A. The construction underlies a  correspondence  between
hyper-K\"ahler and hyper-symplectic structures in dimension four.
\end{abstract}

\keywords{hyper-symplectic structures, hyper-parak\"ahler metrics,
hyper-K\"ahler metrics}
\subjclass[2000]{58J60, 53C26, 53C50}

\title[Bianchi type A hyper-symplectic and
 hyper-K\"ahler metrics in 4D]{Bianchi type A hyper-symplectic and
 hyper-K\"ahler metrics in 4D}

\date{\today }

\author{L.C. de Andr\'es}
\address[Luis C. de Andr\'es, Marisa Fern\'andez, Jos\'e A. Santisteban]{%
Universidad del Pa\'{\i}s Vasco\\
Facultad de Ciencia y Tecnolog\'{\i}a, Departamento de Matem\'aticas\\
Apartado 644, 48080 Bilbao\\
Spain}
\email{luisc.deandres@ehu.es}
\email{marisa.fernandez@ehu.es}
\email{joseba.santisteban@ehu.es}
\author{M. Fern\'andez}
\author{S. Ivanov}
\address[Stefan Ivanov]{University of Sofia, Faculty of Mathematics and
Informatics, blvd. James Bourchier 5, 1164, Sofia, Bulgaria}
\address{and Department of Mathematics, University of Pennsylvania,
Philadelphia, PA 19104-6395}
\email{ivanovsp@fmi.uni-sofia.bg}
\author{J.A. Santisteban}
\author{L. Ugarte}
\address[Luis Ugarte]{Departamento de Matem\'aticas\,-\,I.U.M.A.\\
Universidad de Zaragoza\\
Campus Plaza San Francisco\\
50009 Zaragoza, Spain}
\email{ugarte@unizar.es}
\author{D. Vassilev}
\address[Dimiter Vassilev]{ Department of Mathematics and Statistics\\
University of New Mexico\\
Albuquerque, New Mexico, 87131-0001}
\email{vassilev@math.unm.edu}
\maketitle
\tableofcontents

\setcounter{tocdepth}{2}

\section{Introduction}

Manifolds carrying a hyper-symplectic structure are the pseudo-Riemannian
counterpart to hyper-K\"ahler manifolds. They are defined as follows. An
\emph{almost hyper-paracomplex structure} on a $4n$-dimensional manifold $M$
is a triple $(J,P_1,P_2)$ of anticommuting endomorphisms of the tangent
bundle of $M$ satisfying the paraquaternionic identities
\begin{equation*}
J^2=-P_1^2=-P_2^2=-Id, \quad JP_1=P_2.
\end{equation*}
In this case, $(M,J,P_1,P_2)$ is said to be an \emph{almost
hyper-paracomplex manifold}. Moreover, if $g$ is a pseudo-Riemannian metric
on $(M,J,P_1,P_2)$ for which $J$ is an isometry while $P_1$ and $P_2$ are
anti-isometries, then $g$ is called an \emph{almost hyper-parahermitian
metric}. Necessarily $g$ is a neutral metric, that is, of signature $(2n,2n)$. Such a metric gives rise to three $2$-forms on $M$ defined in {a way similar} to
 the K\"ahler forms in the positive definite case (see Sections 2 and 3
for details). When {these three forms} are closed, the structure $(g,J,P_1,P_2)$ is called
\emph{hyper-symplectic} \cite{Hit1}. {In this case the
structures $(J,P_1,P_2)$ are parallel with respect to the
Levi-Civita connection of the neutral metric $g$ \cite{AH,Hit1,DJS} and, in particular, integrable}.
Metrics associated to a hyper-symplectic structure are also called neutral
hyper-K\"ahler \cite{Kam}, para-hyperk\"ahler ~\cite{BV}, hyper-parak\"ahler
\cite{IZ}, etc.

Manifolds with a hyper-symplectic structure have also a rich geometry.
Indeed, the neutral metric is K\"ahler, Ricci flat and its holonomy group is
contained in $Sp(2n,\mathbb{R})$ \cite{Hit1}. Furthermore, in dimension
four, any hyper-symplectic structure
underlies
an anti-self-dual  and Ricci-flat neutral metric. For this
reason such structures have been used in string theory \cite%
{OV,hul,JR,Bar,Hull,CHO,GL} and integrable systems \cite{D12,BM,DW}. {\ In
fact, in \cite{OV} is considered $N=2$ superstring theory, and it is proved
that the critical dimension of such a string is $4$ and that the bosonic
part of the $N=2$ theory corresponds to self-dual metrics of signature $%
(2,2) $.}

There are not known many explicit hyper-symplectic metrics. In dimension
four they were described in \cite{DS,T,KM,GL}. In higher dimensions,
hyper-symplectic structures on a class of compact quotients of $2$-step
nilpotent Lie groups were exhibited in \cite{FPPS}. Also, in \cite{AD} is
given a procedure to construct hyper-symplectic structures on $\mathbb{R}%
^{4n}$ with complete and not necessarily flat associated neutral metrics.

On the other hand, Bourliot, Estes, Petropoulos and Spindel in \cite{BEPS2}
(see also  \cite{BEPS1})
show a classification of self-dual four-dimensional  gravitational instantons.
Their classification
is based on the algebra homomorphisms relating the Bianchi
group and the group $SO(3)$.
In \cite{BEPS1} it is noticed that the groups $VI_{0}$ and $VIII$
of Bianchi classification seem not interesting for gravitational instantons taking into account
their complex nature, but they may play a more physical role in the context of
neutral hyper-K\"ahler metrics.

In this paper, following an idea originally proposed by Hitchin
\cite{Hitt}, we {re}-construct in a simple way the well known
(cohomogeneity-one) hyper-K\"ahler metrics in dimension four arising from the
three-dimensional groups of Bianchi type A and give explicit
construction of hyper-symplectic metrics of signature (2,2) (that
is, of signature $(+,+,-,-)$) some of which seem to be new. We
follow  Hitchin's idea \cite{Hitt} reducing the problem to a
solution of a certain system of evolution equations applying two
results of Hitchin, which assert that an almost hyper-Hermitian
structure (resp. an almost para-hyperhermitian structure) is
hyper-K\"ahler (resp. hyper-symplectic) exactly when the three
K\"ahler forms are closed \cite{Hit} (resp. \cite{Hit1}). In this
way we recover some of the known (Bianchi-type) hyper-K\"ahler metrics in
dimension four
\cite{LP1,LP2,Lor,EH,GH,BGPP,GPop,Gh1,Gh2,GR,Bar,NP,SD,Cv04,VY}.
Our approach seems to be particularly simple and leads quickly to
the explicit form of the considered metrics. In \cite{KM} it was
discovered a correspondence between Bianchi type $IX$
hyper-K\"ahler metric and Bianchi type VIII hyper-symplectic
metric. Our construction shows that the system of evolution
equations describing the hyper-K\"ahler metrics of Bianchi types
$IX,VIII,VII_0,VI_0,II$ coincides with the system describing
hyper-symplectic metrics of Bianchi types $VIII,IX,VI_0,VII_0,II$,
respectively. This extends the correspondence found in \cite{KM}
between hyper-K\"ahler and hyper-symplectic metrics arising from
three dimensional groups $SU(2)$ and $SU(1,1)$ to a correspondence
between hyper-K\"ahler and hyper-symplectic metrics  arising from
the three dimensional groups of Lorentzian and Euclidean motions
and the Heisenberg group, respectively.

\begin{conv}
The triple $(i,j,k)$ denotes any cyclic permutation of $(1,2,3)$.
\end{conv}

\textbf{Acknowledgments} The research was initiated during the
visit of the third author to the Abdus Salam ICTP, Trieste as a
Senior Associate, Fall 2008. He also thanks ICTP for providing the
support and an excellent research environment. S.I. is partially
supported by the Contract 181/2011 with the University of Sofia
`St.Kl.Ohridski'. S.I and D.V. are partially supported by Contract
``Idei", DO 02-257/18.12.2008 and DID 02-39/21.12.2009. This work
has been also partially supported through grant MCINN (Spain)
MTM2008-06540-C02-01/02. Thanks are due also to the referees for valuable comments improving the clarity of the paper.

\section{Hyper-K\"ahler metrics in dimension four}

In this section we recover some of the known hyper-K\"ahler metrics in dimension four. To this end, we lift the special structure on
the non-Euclidean Bianchi type groups of class A to a hyper-K\"ahler metric
on its product with (an interval in) the real line.

Let us recall firstly that an \emph{almost hypercomplex structure} on a $4n$
dimensional manifold $M$ is a triple $(J_1,J_2,J_3)$ of almost complex
structures on $M$ satisfying the quaternionic identities, that is,
\begin{equation*}
J_1^2=J_2^2=J_3^2=-1, \quad J_1J_2=-J_2J_1=J_3.
\end{equation*}
An almost hypercomplex manifold $(M,J_1,J_2,J_3)$ is said to be \emph{almost
hyper-Hermitian} if there exists a Riemannian metric $g$ on $M$ for which
each of the almost complex {structures}
$J_s$, $s=1,2,3$, is an isometry,
that is, they satisfy the compatibility conditions
\begin{equation*}
g(J_s\cdot,J_s\cdot)=g(\cdot,\cdot), \quad s=1,2,3.
\end{equation*}
In this case, we can define the  fundamental 2-forms $F_s$ of the almost
hyper-Hermitian manifold $(M,g,J_1,J_2,J_3)$ by
\begin{equation}\label{e:fund 2-forms}
F_s(\cdot,\cdot)=g(J_s\cdot,\cdot),\quad s=1,2,3.
\end{equation}
When the structures $J_s$, $s=1,2,3$, are parallel with respect to the
Levi-Civita connection, $(M,g,J_1,J_2,J_3)$ is said to be a \emph{%
hyper-K\"ahler} manifold.
A result of Hitchin~\cite{Hit} states that the fundamental $2$-forms are
closed exactly when the almost hyper-Hermitian structure is hyper-K\"ahler.

Let $G$ be a three dimensional Lie group and let $\{e^1(t),e^2(t),e^3(t)\}$ be a global basis of 1-forms on $G$
for each $t\in I$, where $I\subset \mathbb{R}$ is a connected interval in the real line.
We consider the almost hyper-Hermitian structure on $G\times I$ defined by the following fundamental 2-forms
\begin{equation}  \label{4inst}
\begin{aligned}
F_1=e^1(t)\wedge e^2(t)+e^3(t)\wedge f(t)dt,\\
F_2=e^1(t)\wedge e^3(t)-e^2(t)\wedge f(t)dt,\\
F_3=e^2(t)\wedge e^3(t)+e^1(t)\wedge f(t)dt,
\end{aligned}
\end{equation}
where $f(t)$ is function of $t\in I$ which does not vanish. {Using $\{e^1(t),e^2(t),e^3(t), f(t)dt \}$ as a positively oriented orthonormal basis the fundamental 2-forms are self-dual (SD). When $M$ is hyper-K\"ahler it is necessarily anti-self-dual (ASD) and Ricci flat.}
With the help of Hitchin's theorem \cite{Hit}, it is straightforward to
prove the next basic for our purposes result.

\begin{prop}\label{p:hitchin evolution}
\label{hitsu2} The almost hyper-Hermitian structure $(F_1,F_2,F_3)$ is
hyper-K\"ahler if and only if
\begin{equation}  \label{4inst1}
de^{12}(t_0)=de^{13}(t_0)=de^{23}(t_0)=0,
\end{equation}
for some $t_0\in I$,
and the following evolution equations hold
\begin{equation} \label{evol4}
\frac{\partial}{\partial t}e^{ij}(t)=-f(t)de^k(t).
\end{equation}
Here, $\{i,j,k\}$ denotes an even permutation of $\{1,2,3\}$ and $e^{ij}(t)=
e^i(t)\wedge e^j(t)$.

The hyper-K\"ahler metric is given by
\begin{equation}  \label{hypkel4}
g=(e^1(t))^2+(e^2(t))^2+(e^3(t))^2+f^2(t)dt^2.
\end{equation}
\end{prop}

\begin{proof}
Taking the exterior derivatives in \eqref{4inst} and separating the
variables, we obtain
\begin{equation}  \label{der}
\begin{aligned}
dF_1=de^{12}(t)+\Big[\frac{\partial}{\partial
t}e^{12}(t)+f(t)de^3(t)\Big]\wedge dt,\\
dF_2=de^{13}(t)+\Big[\frac{\partial}{\partial
t}e^{13}(t)-f(t)de^2(t)\Big]\wedge dt,\\
dF_3=de^{23}(t)+\Big[\frac{\partial}{\partial
t}e^{23}(t)+f(t)de^1(t)\Big]\wedge dt.
\end{aligned}
\end{equation}
The equations \eqref{der} imply that
all the three 2-forms $F_i$ are closed precisely when
\eqref{evol4} holds and $de^{12}(t)=de^{13}(t)=de^{23}(t)=0$ for all $t\in I$.
The latter condition is equivalent to \eqref{4inst1} because
by \eqref{evol4} we have that $\frac{\partial}{\partial t} ( de^{ij}(t) ) = d(f(t)\, de^k(t)) = f(t)\, d^2e^k(t) =0$
for all $t$, that is, $de^{ij}(t)$ is constant on the connected interval~$I$.
Finally, Hitchin theorem \cite{Hit} completes the proof.
\end{proof}

In the next table we recall the Bianchi classification \cite{Ketal} (up to isomorphism) of the three dimensional Lie algebras using $de^{1}=-b_{1}e^{23}$,
$de^{2}=-a\,e^{12}-b_{2}e^{31}$, and $de^{3}=-b_{3}e^{12}+a\,e^{31}$.
\medskip

\newpage

\begin{table}[h]
{\setlength{\extrarowheight}{3pt}\tiny
\begin{tabular}{|c|c|c|c|c|c|c|c|c|c|c|c|}
\hline
Type & $I$ & $II$ & $VI_{0}$ & $VII_{0}$ & $VIII$ & $IX$ & $V$ & $IV$ & $VII_a, a> 0$ & $III$
&$VI_a, 0<a\not=1$ \\ \hline
$a$ & $0$ & $0$ & $0$ & $0$ & $0$ & $0$ & $1$ & $1$ & $a$ & $1$ & $a$ \\
\hline
$b_{1}$ & $0$ & $0$ & $0$ & $0$ & $1$ & $1$ & $0$ & $0$ & $0$ & $0$ & $0$ \\
\hline
$b_{2}$ & $0$ & $0$ & $1$ & $1$ & $1$ & $1$ & $0$ & $0$ & $1$ & $1$ & $1$
\\ \hline
$b_{3}$ & $0$ & $1$ & $-1$ & $1$ & $-1$ & $1$ & $0$ & $1$ & $1$ & $-1$ & $-1$
\\ \hline
$G$& Abelian & Heisenberg & Poincare a.k.a.& Euclidean & Lorentz & Rotations
&  &  &  &  & \\
&  & a.k.a. Nil & Lorentzian motions & motions & $SU(1,1)$ & $SU(2)$ &  &  &  &  &  \\ \hline
& \multicolumn{6}{|c|}{type A (unimodular)} & \multicolumn{5}{|c|}{type B
(non-unimodular)} \\ \hline
\end{tabular}\vspace{5pt}
\caption{Bianchi classification}}
\end{table}

{Our goal is to seek the explicit solution of the system given in the
previous proposition for each of the three dimensional Bianchi type groups. We will consider the class of left-invariant evolutions, that is to say, $\{e^1(t),e^2(t),e^3(t)\}$ is a basis
of $\mathfrak{g}^*$ for all $t\in I$, $\mathfrak{g}$ being the Lie algebra of $G$, i.e., for some function $c^i_j(t)$ we have
\begin{equation}\label{e:gen evol}
e^i(t) = c^i_1(t)\, e^1 + c^i_2(t)\, e^2 + c^i_3(t)\, e^3, \quad i=1,2,3,
\end{equation}
with four-dimensional metric \eqref{hypkel4} taking the form $g=f^2(t)\,dt^2\ +\ c^k_i(t)\,c^k_j(t)\,  e^i\odot e^j $. This special class of evolutions restricts the applicability of
our method to the groups of type A.}

\begin{lemma}\label{left-invariant evolution}
Let $\{e^1(t),e^2(t),e^3(t)\}$ be a basis of left-invariant 1-forms on $G$
for each $t\in I$. Then, \eqref{4inst1} is satisfied if and only if
the Lie group $G$ is of Bianchi type A.
\end{lemma}

\begin{proof}
Let us fix the basis $\{e^1(t_0),e^2(t_0),e^3(t_0)\}$ of $\mathfrak{g}^*$. Then, there are
constants $c^i_j\in\mathbb{R}$ such that
$$
e^i(t_0)= c^i_1\, e^1 + c^i_2\, e^2 + c^i_3\, e^3, \quad i=1,2,3,
$$
where $\{e^1,e^2,e^3\}$ denotes the basis of $\mathfrak{g}^*$ given in Table~1, i.e.
satisfying $de^{1}=-b_{1}e^{23}$,
$de^{2}=-a\, e^{12}-b_{2}e^{31}$ and $de^{3}=-b_{3}e^{12}+a\, e^{31}$. Thus
$e^{ij}(t_0)= c^{ij}_{12} e^{12} + c^{ij}_{13} e^{13} + c^{ij}_{23} e^{23}$, where $c^{ij}_{rs}= c^i_r c^j_s - c^i_s c^j_r$.

Since $de^{12}=de^{13}=0$ and $de^{23}=-2a\, e^{123}$, we have that condition \eqref{4inst1} is equivalent to
$$
0= de^{ij}(t_0)= c^{ij}_{23} de^{23} = -2a\, c^{ij}_{23} e^{123}
$$
for $(i,j)=(1,2)$, $(1,3)$ and $(2,3)$. Now, $G$ is of Bianchi type B if and only if $a\not=0$,
and therefore \eqref{4inst1} is satisfied if and only if $c^{12}_{23}=c^{13}_{23}=c^{23}_{23}=0$. But, the latter condition implies $\det (c^i_j) =0$, which contradicts the fact that $\{e^1(t_0),e^2(t_0),e^3(t_0)\}$ is a basis.

\end{proof}

 {In the following sub-sections for each of the Bianchi type A groups we explicitly solve  the system of Proposition \ref{p:hitchin evolution} under the restrictive assumption that the evolution of the invariant one forms is diagonal. Specifically, given a
basis $\{e^1,e^2,e^3\}$ of $\mathfrak{g}^*$, we will consider the "diagonal" evolution
\begin{equation}  \label{ev}
e^1(t)=f_1(t)e^1, \quad\quad e^2(t)=f_2(t)e^2, \quad\quad e^3(t)=f_3(t)e^3,
\end{equation}
where $f_1,f_2,f_3$ are non-vanishing functions of $t\in I$. The corresponding "diagonal" metric is given by
\begin{equation}\label{e:tri-axialform}
g=f_1^2(e^1)^2+f_2^2(e^2)^2+f_3^2(e^3)^2+f^2dt^2.
\end{equation}
The function $f$ is introduced for convenience in order to identify the metrics obtained by our method with the known explicit examples of four dimensional hyper-K\"ahler metrics.
We note explicitly that according  to \cite{T} a cohomogeneity-one Einstein metric whose general form is
\[
g=dT^2 + h_{ij}(T)\, e^i\odot e^j
\]
can be assumed to be  diagonal, $h_{ij}=0$ for $i\not= j$, when $\{e^i\}$ are the left invariant forms of a Bianchi type VIII or IX group - first one diagonalizes $( h_{ij})$ at a time $T_0$ and then show that the Ricci flatness implies the vanishing of all off-diagonal terms.  On the hand a diagonalization of $( h_{ij})$ corresponds to a rotation of the matrix $(c^i_j)$ which in turn becomes a rotation of the fundamental 2-forms \eqref{e:fund 2-forms}. Such a rotation might lead to non-closed 2-forms. Thus our assumption of a diagonal evolution is restrictive in all cases. Our next task is to consider case by case the Bianchi type A groups.}

\subsection{The group $SU(2)$, Bianchi type $IX$}

Let $G=SU(2)=S^3$ be described by the structure equations
\begin{equation}  \label{su2}
de^i=-e^{jk}.
\end{equation}
In terms of Euler angles, the left invariant 1-forms $e^i$ are given by
\begin{equation}  \label{eulersu2}
\begin{aligned} e^1&=&\sin\psi d\theta-\cos\psi\sin\theta d\phi,\quad
e^2&=&\cos\psi d\theta+\sin\psi\sin\theta d\phi,\quad e^3&=&d\psi
+\cos\theta d\phi. \end{aligned}
\end{equation}
We evolve the $SU(2)$ structure as in \eqref{ev}.
Using \eqref{su2} we reduce the evolution equations \eqref{evol4} to the
following system of ODEs
\begin{equation}  \label{evol4f}
\frac{d}{d t}(f_1f_2)=ff_3,\quad \frac d{dt}(f_1f_3)=ff_2, \quad \frac
d{dt}(f_2f_3)=ff_1.
\end{equation}
The system \eqref{evol4f} is equivalent to the following
"BGPP" system (see \cite{BGPP})
\begin{equation}  \label{ah}
\frac {d}{dt}f_1=f\frac{f_2^2+f_3^2-f_1^2}{2f_2f_3}, \quad \frac {d}{dt}f_2=f%
\frac{f_3^2+f_1^2-f_2^2}{2f_1f_3}, \quad \frac {d}{dt}f_3=f\frac{%
f_1^2+f_2^2-f_3^2}{2f_1f_2}.
\end{equation}
The system \eqref{ah} admits the triaxial Bianchi IX BGPP \cite{BGPP}
hyper-K\"ahler metrics by taking $f=f_1f_2f_3$ and all $f_i$
different (see also \cite{GPop}), and the Eguchi-Hanson \cite{EH}
hyper-K\"ahler metric by letting two of the functions $f_1,\, f_2, \, f_3$
equal to each other.

\subsubsection{The general solution}\label{The general solution}
With the substitution $x_{i}=(f_{j}f_{k})^{2}$, the system \eqref{evol4f}
becomes
\begin{equation*}
\frac{dx_{i}}{dr}=2(x_{1}x_{2}x_{3})^{1/4},
\end{equation*}
in terms of the parameter $dr=fdt$. Hence the functions $x_{i}$ differ by a
constant, i.e., there is a function $x(r)$ such that $%
x(r)=x_{1}+a_{1}=x_{2}+a_{2}=x_{3}+a_{3}$. The equation for $x(r)$ is
\begin{equation}
\frac{d{x}}{dr}=2\left( (x-a_{1})(x-a_{2})(x-a_{3})\right) ^{1/4},\ \text{
i.e.\ }dr=\frac{1}{2}\frac{1}{\left( (x-a_{1})(x-a_{2})(x-a_{3})\right)
^{1/4}}dx.  \label{e:BIX1}
\end{equation}
If we let $h(x)=\frac{1}{2}\left( (x-a_{1})(x-a_{2})(x-a_{3})\right) ^{-1/4}$
and take into account $x_{i}=(f_{j}f_{k})^{2}$, we see from \eqref{evol4f}
that the functions $f_{i}(x)$ satisfy
\begin{equation*}
\frac{d}{dx}\left( (x-a_{i})^{1/2}\right) =h(x)f_{i}.
\end{equation*}
\noindent Solving for $f_{i}$ we showed that the general solution of \eqref{evol4f} is
\begin{equation}  \label{e:su2}
f_{i}(x)=\frac{(x-a_{j})^{1/4}(x-a_{k})^{1/4}}{(x-a_{i})^{1/4}}
,\quad f(t)=h\left ( x (t)\right )\,x^{\prime }(t), \quad
h(x)=\frac{1}{2}\left( (x-a_{1})(x-a_{2})(x-a_{3})\right) ^{-1/4},
\end{equation}
where $a_{1}, a_{2}$ and $a_{3}$ are constants, and $x$ is an auxiliary
independent variable (substituting any function $x=x(t)$ gives a solution of
\eqref{evol4f} in terms of $t$ in an interval where $f$ and $f_i$, with $i=1,2,3$,
do not vanish).

The resulting hyper-K\"ahler metric is given by \eqref{e:tri-axialform} where
the forms $e^1,e^2,e^3$ and the functions $f_1,f_2,f_3,f$ are
given by \eqref{eulersu2} and \eqref{e:su2}, respectively.

\subsubsection{Eguchi-Hanson instantons}

A particular solution to \eqref{evol4f} is obtained by taking $x=(t/2)^4$
and $a_1=a_2=\frac1{16}a$, $a_3=0$, which gives
\begin{equation}  \label{ehfff}
f_1=f_2=\frac t2,\quad f_3=\frac {1}{2t} (t^{4}-a)^\frac{1}{2}, \quad f=%
\sqrt{\frac{t^4}{t^4-a}},
\end{equation}
where $t\in (0, \infty)$ if $a\leq 0$, and $t\in (a^{\frac{1}{4}}, \infty)$ if $a>0$.
This is the Eguchi-Hanson instanton \cite{EH} with the metric given by
\begin{equation*}
g=\frac{t^2}4\Big((e^1)^2+(e^2)^2\Big)+\frac{t^4-a}{4t^2} (e^3)^2 + \frac{t^4}{t^4-a} (dt)^2,
\end{equation*}
where the 1-forms $e^1,e^2,e^3$ can be found in \eqref{eulersu2}.

\subsubsection{Triaxial Bianchi type $IX$ BGPP metrics \protect\cite{BGPP}}

If we perform the substitution $x=t^4$, $a_1=a^4$, $a_2=b^4$ and $a_3=c^4$ the metric of \eqref{e:su2} turns into the triaxial Bianchi $IX$ metrics discovered in \cite{BGPP} (see
also \cite{GPop,Gh1,Gh2})
\begin{equation}  \label{triax}
\begin{aligned}
&f_1(t)=\frac{(t^4-b^4)^\frac14(t^4-c^4)^\frac14}{(t^4-a^4)^\frac14},\qquad
f_2(t)=\frac{(t^4-a^4)^\frac14(t^4-c^4)^\frac14}{(t^4-b^4)^\frac14},\\
&f_3(t)=\frac{(t^4-a^4)^\frac14(t^4-b^4)^\frac14}{(t^4-c^4)^\frac14},\qquad
f(t)=\frac{2t^3}{(t^4-b^4)^\frac14(t^4-c^4)^\frac14(t^4-a^4)^\frac14}.
\end{aligned}
\end{equation}
{By \cite{BGPP} the above metrics are all singular (cannot be completed) with the only exception being the Euclidean flat space (when $a_1=a_2=a_3$) and the Eguchi-Hanson instanton, whose completion is the cotangent bundle of the complex  projective line $\mathbb{C}P^1$. The latter example was extended to higher dimensions by Calabi \cite{Cal}.  In the derivation above we avoided the use of elliptic functions.}

{
We note that the Atiyah-Hitchin class of complete hyper-K\"ahler metrics is not included in our derivation. In fact, by construction, the natural action of $SU(2)$ extends to a trivial action on the fundamental 2-forms, while it is known that in the Atiyah-Hitchin class  the group $SU(2)$ rotates the self-dual forms \cite{AH}.}

\subsection{The group $SU(1,1)$, Bianchi type $VIII$} (Bianchi type $VIII$ were investigated in \cite{LP1,LP2,Lor}.)
Let $G=SU(1,1)$ be described by the structure equations
\begin{equation}  \label{su11}
de^1=-e^{23}, \quad de^2=-e^{31}, \quad de^3=e^{12}.
\end{equation}
In terms of local coordinates the left invariant forms $e^i$ are given by
\begin{equation}  \label{su11c}
\begin{aligned} e^1&=&d\psi -\cos\theta d\phi,\quad e^2&=&\sinh\psi
d\theta+\cosh\psi\sin\theta d\phi, \quad e^3&=&\cosh\psi
d\theta+\sinh\psi\sin\theta d\phi. \end{aligned}
\end{equation}
We evolve the $SU(1,1)$ structure as
in \eqref{ev}. Using the structure equations \eqref{su11c}, we reduce the
evolution equations \eqref{evol4} to the following system of ODEs
\begin{equation}  \label{evol4f11}
\frac{d}{dt}(f_1f_2)=-ff_3,\quad \frac{d}{dt} (f_1f_3)=ff_2, \quad \frac{d}{%
dt}(f_2f_3)=ff_1.
\end{equation}
Solutions to the above system yield the hyper-K\"ahler metrics \eqref{hypkel4} found in \cite{BGPP}.

\subsubsection{Triaxial Bianchi type $VIII$ metrics.}

Working as in \ref{The general solution}, we obtain the next system for the
functions $x_i$
\begin{equation*}
\frac {dx_2}{dr} =\frac {dx_1}{dr}= 2(x_1x_2x_3)^{1/4}, \qquad \frac {dx_3}{%
dr}= -2(x_1x_2x_3)^{1/4}.
\end{equation*}
Solving for $f_{i}$, as in the derivation \eqref{e:su2}, we find that the
general solution of \eqref{evol4f11} is
\begin{equation}  \label{e:su2g}
\begin{aligned}
&f_{1}(x)=\frac{(x-a_{2})^{1/4}(a_{3}-x)^{1/4}}{(x-a_{1})^{1/4}}, \qquad
f_{2}(x)= \frac{(x-a_{1})^{1/4}(a_{3}-x)^{1/4}}{(x-a_{2})^{1/4}},\\
&f_{3}(x)=\frac{(x-a_{1})^{1/4}(x-a_{2})^{1/4}}{(a_{3}-x)^{1/4}},\quad
f(t)=h\left ( x (t)\right )\,x'(t),\quad h(x)=\frac{1}{2}\left(
(x-a_{1})(x-a_{2})(a_{3}-x)\right) ^{-1/4}, \end{aligned}
\end{equation}
where $a_{1},a_{2}$ and $a_{3}$ are constants, and $x$ is an auxiliary
independent variable (substituting any function $x=x(t)$ gives a solution of %
\eqref{evol4f} in terms of $t$ in an interval where $f$ and $f_i$,
$i=1,2,3$, do not vanish).

The resulting hyper-K\"ahler metric is given by \eqref{e:tri-axialform} where the forms $e^1,e^2,e^3$ and the functions $f_1,f_2,f_3,f$ are given by \eqref{su11c} and \eqref{e:su2g}, respectively.

Taking $f=f_1f_2f_3$ and all $f_i$ different into \eqref{e:tri-axialform}, we obtain explicit expression
of the triaxial Bianchi $VIII$ solutions indicated in \cite{BGPP}.

A particular solution is obtained by letting $a_1=a_2=0, a_3=\frac{a}{16}, x=(t/2)^4$
which gives
\begin{equation*}
f_1=f_2=\frac12(a-t^4)^{\frac14}, \quad f_3=\frac{t^2}2(a-t^4)^{-\frac14},
\quad f=t(a-t^4)^{-\frac14},
{\quad 0<t^4<a.}
\end{equation*}
The resulting {hyper-K\"ahler} metric is
\begin{equation*}
g=\frac{\sqrt{a-t^4}}{4}\left((e^{1})^2+(e^{2})^2 \right)+\frac{t^4}{4\sqrt{a-t^{4}}}(e^{3})^2+\frac{t^2}{\sqrt{a-t^{4}}}dt^2,
\end{equation*}
\noindent where the forms $e^i$ are given by \eqref{su11c}.

\subsection{The Heisenberg group $H^3$, Bianchi type $II$, Gibbons-Hawking
class}

Consider the two-step nilpotent Heisenberg group $H^3$ defined by the
structure equations
\begin{equation}  \label{heis3}
\begin{aligned} de^1=de^2=0, \qquad de^3=-e^{12}, \end{aligned}
\end{equation}
where we can consider
\begin{equation}\label{heis33}
e^1=dx, \quad e^2=dy, \quad e^3=dz-\frac12xdy+\frac12ydx,
\end{equation}
with $x,y,z$ the global coordinates functions on $H^3$.
Now, we evolve the
structure of $H^3$ according to \eqref{ev}. The structure equations %
\eqref{heis3} reduce the evolution equations \eqref{evol4} to the following
system of ODEs
\begin{equation}  \label{evol4h}
\frac{d}{dt}(f_1f_2)=ff_3,\quad \frac{d}{dt} (f_1f_3)=0, \quad \frac{d}{dt}%
(f_2f_3)=0.
\end{equation}
Working as in the previous example, i.e., using the same substitutions we
see that the function $x_i$ satisfy
\begin{equation*}
\frac {dx_3}{dr} = 2(x_1x_2x_3)^{1/4},\qquad \frac {dx_1}{dr} = \frac {dx_2}{%
dr}=0.
\end{equation*}
The general solution of this system is
\begin{equation}  \label{e:H}
x_1=a, \qquad x_2=b, \qquad x_3=\left (\frac 32 (ab)^{1/4}\, r +c \right)
^{4/3},
\end{equation}
where $a, b$ and $c$ are constants. Therefore, using again $f_i=\left (
\frac {x_jx_k}{x_i}\right )^{1/4} $, the general solution of \eqref{evol4h}
is
\begin{equation}  \label{e:H sols for xi}
\begin{aligned}
& f_1= \left (\frac ba \right )^{1/4}\left (\frac 32
(ab)^{1/4}\, r +c \right) ^{1/3}, \quad f_2= \left (\frac ab \right
)^{1/4}\left (\frac 32 (ab)^{1/4}\, r +c \right) ^{1/3},\\
& f_3= \frac
{\left (ab \right ) ^{1/4}}{\left (\frac 32 (ab)^{1/4}\, r +c \right )
^{1/3}}, \quad f=r'(t),  \end{aligned}
\end{equation}
where  $r=r(t)$ is an arbitrary function for $t$ in an interval where $f$ and $f_i$, $i=1,2,3$ do not vanish.

The resulting hyper-K\"ahler metric is given by \eqref{e:tri-axialform} where the forms $e^1,e^2,e^3$ and the functions $f_1,f_2,f_3,f$ are given by \eqref{heis33} and \eqref{e:H sols for xi}, respectively.

A particular solution is obtained by taking $c=0$ and $a=b=1$ into \eqref{e:H sols for xi}, which gives
\begin{equation*}
f_1=f_2=\lambda r^{1/3}, \qquad f_3=f_1 ^{-1},
\end{equation*}
with $\lambda=\left (\frac 32 \right )^{1/3}$. The substitution
$t=\lambda^2 r^{2/3}$ gives $f_1=f_2=f=t^{\frac12}, \quad
f_3=t^{-\frac12}$, with $t>0$. This is the hyper-K\"ahler metric,
first written in \cite{Lor,LP1},
\begin{equation*}
g=t\Big(dt^2+dx^2+dy^2\Big)+\frac1t\Big(dz-\frac12xdy+\frac12ydx\Big)^2,
\end{equation*}
belonging to the Gibbons-Hawking class \cite{GH} with an $S^1$-action and
known also as Heisenberg metric \cite{GR} (see also \cite{Bar,NP,SD,Cv04,VY}).

\subsection{Rigid motions of euclidean 2-plane-Bianchi $VII_0$}

We consider the group $E_2$ of rigid motions of Euclidean 2-plane defined by
the structure equations
\begin{equation}  \label{euc}
\begin{aligned} &de^1=0, \quad de^2=e^{13},\quad de^3=-e^{12},
\text{\phantom {1}where we can consider}\\ &e^1=d\phi, \quad e^2=\sin\phi\,
dx-\cos\phi\, dy, \quad e^3 =\cos\phi\, dx+\sin\phi\, dy. \end{aligned}
\end{equation}
We evolve the structure as in %
\eqref{ev}. Using the structure equations \eqref{euc} we reduce the
evolution equations \eqref{evol4} to the following system of ODE
\begin{equation}  \label{evol4feuc}
\frac{d}{dt}(f_1f_2)=ff_3,\quad \frac{d}{dt} (f_1f_3)=ff_2, \quad \frac{d}{dt%
}(f_2f_3)=0.
\end{equation}
With the substitution $x_{i}=(f_{j}f_{k})^{2}$, the above system becomes
\begin{equation*}
\frac{dx_{1}}{dr}=0,\qquad \frac{dx_{2}}{dr}=\frac{dx_{3}}{dr}%
=2(x_{1}x_{2}x_{3})^{1/4},
\end{equation*}
in terms of the parameter $dr=fdt$. Hence, there is a function $x(r)$ and
three constants $a_1, \ a_2, \ a_3$, such that, $%
x(r)=x_{2}+a_{2}=x_{3}+a_{3} $, $x_1=a_1$. The equation for $x(r)$ is
\begin{equation}
\frac{d{x}}{dr}=2\left( a_{1}(x-a_{2})(x-a_{3})\right) ^{1/4},\ \text{
i.e.,\ }dr=\frac{1}{2}\frac{1}{\left( a_{1}(x-a_{2})(x-a_{3})\right) ^{1/4}}%
dx.  \label{e:BIX12}
\end{equation}
If we let $h(x)=\frac{1}{2}\left( a_{1}(x-a_{2})(x-a_{3})\right) ^{-1/4}$,
and take into account $x_{i}=(f_{j}f_{k})^{2}$, we see from \eqref{evol4feuc}
that the functions $f_{i}(x)$ satisfy
\begin{equation*}
\frac{d}{dx}\left( (x-a_{i})^{1/2}\right) =h(x)f_{i}, \quad i=2, \ 3.
\end{equation*}
Solving for $f_{i}$ we show that the general solution of \eqref{evol4f} is
\begin{equation}  \label{e:Emot}
\begin{aligned} &f_{1}(x)=\frac{(x-a_{2})^{1/4}(x-a_{3})^{1/4}}{a_{1}^{1/4}}
,\qquad f_{2}(x)=\frac{a_{1}^{1/4}(x-a_{3})^{1/4}}{(x-a_{2})^{1/4}},\qquad
f_{3}(x)=\frac{a_{1}^{1/4}(x-a_{2})^{1/4}}{(x-a_{3})^{1/4}},\\&  f(t)=
h\left ( x (t)\right )\,x'(t),\qquad h(x)=\frac{1}{2}\left(
(a_{1}(x-a_{2})(x-a_{3})\right) ^{-1/4}, \end{aligned}
\end{equation}
where $a_{1},\ a_{2}$ and $a_{3}$ are constants, and $x$ is an auxiliary
independent variable (substituting any function $x=x(t)$ gives a solution of %
\eqref{evol4feuc} in terms of $t$ in an interval where $f$ and $f_i$, $i=1,2,3$, do not vanish).

The resulting hyper-K\"ahler metric is given by \eqref{e:tri-axialform} where the forms $e^1,e^2,e^3$ and the functions $f_1,f_2,f_3,f$ are given by \eqref{euc} and \eqref{e:Emot}, respectively.

\medskip
A particular vacuum solutions of Bianchi type $VII_0$ is obtained by letting
$f_2=f_3^{-1}$ and $f_1=f$, which gives
\begin{equation*}
\frac{d}{dt}(ff_3^{-1})=ff_3,\qquad \frac{d}{dt}(ff_3)=ff_3^{-1},
\end{equation*}
with general solution of the form $ff_3+ff_3^{-1}=Ae^t$, $ ff_3^{-1}-ff_3=Be^{-t}$. Hence,
\begin{equation*}
f=f_1=\frac12(Ae^t+Be^{-t})^{\frac12}(Ae^t-Be^{-t})^{\frac12},\quad
f_3=f_2^{-1}=(Ae^t+Be^{-t})^{-\frac12}(Ae^t-Be^{-t})^{\frac12},
\end{equation*}
where $t>\frac{1}{2}\log\left|\frac {B}{A}\right|$
since $A^2 e^{2t}-B^2 e^{-2t}>0$.
The resulting hyper-K\"ahler metric is
\begin{equation}  \label{hypeuc}
g=\frac14(A^2e^{2t}-B^2e^{-2t})\Big(dt^2+d\phi^2+%
\frac4{(Ae^t-Be^{-t})^2}(e^2)^2 +\frac4{(Ae^t+Be^{-t})^2}(e^3)^2\Big),
\end{equation}
where $e^2,e^3$ are given by \eqref{euc}.  In particular, setting $A=B$ in \eqref{hypeuc} we obtain
\begin{equation*}
g=\frac{A^2}2\sinh{2t}\Big(dt^2+d\phi^2\Big)+\coth t(e^2)^2 +\tanh t(e^3)^2,
\end{equation*}
which is the vacuum solutions of Bianchi type $VII_0$ \cite{Lor,LP1} with
group of isometries $E_2$ \cite{GR}, (see also \cite{VY}).

\subsection{Rigid motions of Lorentzian 2-plane-Bianchi type $VI_0$}
Now we consider the group of rigid motions $E(1,1)$ of Lorentzian 2-plane
defined by the structure equations and coordinates as follows
\begin{equation}  \label{lor}
\begin{aligned}
& de^1=0, \quad de^2=e^{13}, \quad de^3=e^{12},\\
& e^1=d\phi, \quad e^2=\sinh\phi \, dx+\cosh\phi\, dy, \quad e^3 =\cosh\phi\,
dx+\sinh\phi\, dy. \end{aligned}
\end{equation}
We evolve the structure as in \eqref{ev}. Using the structure equations %
\eqref{lor}, the evolution equations \eqref{evol4} give the following system
of ODEs
\begin{equation}  \label{evol4flor}
\frac{d}{dt}(f_1f_2)=-ff_3,\quad \frac{d}{dt} (f_1f_3)=ff_2, \quad \frac{d}{%
dt}(f_2f_3)=0.
\end{equation}
The general solution of (\ref{evol4flor}) is
\begin{equation}  \label{e:evol4flor}
\begin{aligned}
& f_{1}(x)=\frac{(x-a_{2})^{1/4}(a_{3}-x)^{1/4}}{a_{1}^{1/4}}
,\qquad f_{2}(x)=\frac{a_{1}^{1/4}(a_{3}-x)^{1/4}}{(x-a_{2})^{1/4}},\qquad
f_{3}(x)=\frac{a_{1}^{1/4}(x-a_{2})^{1/4}}{(a_{3}-x)^{1/4}},\\
& f(t)=
h\left ( x (t)\right )\,x'(t),\quad h(x)=\frac{1}{2}\left(
(a_{1}(x-a_{2})(a_{3}-x)\right) ^{-1/4}, \end{aligned}
\end{equation}
where $a_{1},\ a_{2}$ and $a_{3}$ are constants, and $x$ is an auxiliary
independent variable (substituting any function $x=x(t)$ gives a solution of %
\eqref{evol4feuc} in terms of $t$ in an interval where $f$ and $f_i$, $i=1,2,3$, do not vanish).

The resulting hyper-K\"ahler metric is given by \eqref{e:tri-axialform} where the forms $e^1,e^2,e^3$ and the functions $f_1,f_2,f_3,f$ are given by \eqref{lor} and \eqref{e:evol4flor}, respectively.

A particular case is obtained by letting  $f_2=f_3^{-1}$ and $f_1=f$, which turns the system \eqref{e:evol4flor} in the form $\frac{d}{dt} (ff_3^{-1})=-ff_3$,  $\frac{d}{dt}%
(ff_3)=ff_3^{-1}$. {This system is integrated trivially,
\begin{equation*}
ff_{3}^{-1}=a\cos t+b\sin t \quad \text{and}\quad  ff_3=a\sin t-b \cos t,
\end{equation*}
hence
\begin{equation*}
f^{2}=(a\cos t+b\sin t)(a\sin t-b \cos t),\quad f_{3}=(a\sin t-b \cos
t)f^{-1}.
\end{equation*}
Therefore,
\begin{equation*}
f=f_{1}  =(a\cos t+b\sin t)^{1/2}(a\sin t-b \cos t)^{1/2},\quad
f_{2}^{-1}=f_{3}  =(a\cos t+b\sin t)^{-1/2}(a\sin t-b \cos t)^{1/2},
\end{equation*}
and the hyper-K\"ahler metric is given by
\begin{multline}  \label{hyplor}
g= \left (a \sin t -b \cos t\right) \left(a \cos t +b \sin t\right)\Big(dt^2+d\phi^2 \Big )
+\frac{ a \cos t +b \sin t}{a \sin t -b \cos t}(e^2)^2 +\frac{a \sin t -b \cos t }{a \cos t +b \sin t}(e^3)^2,
\end{multline}
where} $e^2$ and $e^3$ are defined in \eqref{lor}. Introducing $t_0$
and $r_0$ by letting $r_0=\sqrt{a^2+b^2}$, $\cos t_0=a/\sqrt{a^2+b^2}$ and $%
\sin t_0=b/\sqrt{a^2+b^2}$,
we have that $t\in (t_0,t_0+\frac{\pi}{2})$, and
the above metric can be put in the form
\begin{equation}  \label{e:hyplor}
g=\frac{1}{2}r_0^2\sin2(t-t_0)(dt^2+d\phi^2)+\cot(t-t_{0})(e^2)^2+
\tan(t-t_{0})(e^3)^2.
\end{equation}
After an obvious re-parametrization the metric \eqref{e:hyplor} takes a familiar form
\begin{equation*}
g=a^2\sin{2\tau}\Big(d\tau^2+d\phi^2\Big)+\cot \tau(e^2)^2 +\tan \tau(e^3)^2,
\end{equation*}
which is the vacuum solutions of Bianchi type $VI_0$ \cite{Lor,LP1} with
group of isometries $E_{(1,1)}$ \cite{GR}, (see also \cite{VY}).

\medskip

As a consequence of the previous subsections we have the following simple fact.
\begin{prop}\label{globconf-to-hyperKahler}
Let $G$ be a 3-dimensional Lie group of Bianchi type A. Then, there exists a
complete metric on $G\times \mathbb{R}$
that is globally conformal to a hyper-K\"ahler metric on $G\times \mathbb{R}$.
\end{prop}
\begin{proof}
In the previous subsections,
for each 3-dimensional Lie group $G$ of Bianchi type A, we have constructed
a hyper-K\"ahler metric of the form $g=g_t+ \left( f(t)dt \right)^2$ on $G\times (a_0,b_0)$,
for some open interval $(a_0,b_0)\subset \mathbb{R}$,
such that the metric $g_t=\sum_{j=1}^3 f_j(t)^2 (e^j)^2$
is a left-invariant metric on $G$ for each $t\in (a_0,b_0)$.
By a suitable change of variables $t=t(r)$ we can assume that $r$ changes from $-\infty$ to $+\infty$ as $t$ changes from $a_o$ to $b_0$. Thus, the considered metrics  take the form $$g=\sum_{j=1}^3 \left( f_j(r) \right)^2 (e^j)^2 + \left (\varphi(r)\right)^2dr^2$$ where $\varphi (r)dr=f(t)dt$, $\varphi>0$, for the fixed $t=t(r)$. This allows to put $g$ in the form $g=\varphi(r)^2\hat g$, where
\begin{equation*}
\hat g=dr^2+\sum_{j=1}^3 \left( f_j/\varphi(r) \right)^2 (e^j)^2=dr^2+g_r.
\end{equation*}
 An adaptation of the proof of completeness of a doubly warped product of complete Riemannian manifolds \cite{ONeill} shows that $\hat g$ is complete. In fact, following \cite{ONeill} consider a Cauchy sequence $\{(p_i,r_i)\}_{i=1}^\infty$ in $G\times \mathbb{R}$. From the form of the metric it follows for any curve $\gamma$ in $G\times \mathbb{R}$ we have $L(\gamma)\geq L(\pi_2\circ\gamma)$ for the corresponding lengths of the curve and its projection, where $\pi_2$ is the projection from $G\times \mathbb{R}$ to $\mathbb{R}$. This implies $$d_{\hat g}\left((p_i,r_i),(p_k,r_k)\right)\geq |r_i-r_{k}|,$$ hence $\{r_i\}$ is a Cauchy sequence in $\mathbb{R}$. Therefore,  the numbers $r_i$ belong to a fixed compact interval, $|r_i|\leq R$. Thus,  $\inf_{|r|\leq R} \left\vert f_j/\varphi \,(r)\right\vert =\sigma>0$ for $j=1,2,3$ and for any curve $\gamma$ with $|\pi_2\circ \gamma|\leq R$ we have  $L(\gamma)\geq \sigma\,L(\pi_1\circ\gamma)$, where $\pi_1$ is the projection from $G\times \mathbb{R}$ to $G$. Since for any curve $\alpha(s)$, $s\in [0,1]$ in $G$ we can consider its  lift to $G\times \mathbb{R}$ defined by
 $$\gamma(s)=\left(\alpha (s), sr'+(1-s)r \right)$$ for which we have $\alpha=\pi_1\circ\gamma$ and $\pi_2\circ\gamma$ has arbitrarily fixed beginning $r$ and end $r'$ with  $|r-r'|\leq R$, it follows
 $$d_{\hat g}\left((p_i,r_i),(p_k,r_k)\right)\geq \sigma\,d_{g_o}(p_i,p_k)$$ where $g_o=\sum_{j=1}^3  (e^j)^2$, hence $\{p_i\}_{i=1}^\infty$ is a Cauchy sequence in $(G,g_o)$ (use $r=r_i$ and $r'=r_k$ in the above constriction). Noting that the metric $g_o$ is a left-invariant metric on the group $G$ it follows that $\{p_i\}_{i=1}^\infty$ is a convergent sequence in $G$. Thus, $\{(p_i,r_i)\}_{i=1}^\infty$ is a convergent sequence in $(G\times \mathbb{R},\hat g)$. By the Hopf-Rinow theorem the latter is a complete Riemannian manifold.
\end{proof}

\subsection{Contractions} {It is worth observing that the hyper-K\"ahler metrics constructed from the Bianchi groups of type $II$, $VI_o$ and $VII_o$  (and the trivial Abelian case) can also be obtained  using the well known contractions of the Lie algebras $su(2)$ (i.e. type IX) and $su(1,1)$ (i.e. type VIII) to any of the former four.  In \cite{ChCGLPW} and \cite{GLPSt} this idea of exploiting Lie algebra contractions was used to construct explicit metrics of special holonomy. For our purposes, consider the contraction corresponding to the following scaling of $su(2)$
\[
e^1=\hat{e}{^1}, \quad e^2=\lambda \hat{e}{^2}, \quad e^3=\lambda \hat{e}{^3},
\]
where $e^1, \, e^2, \, e^3$ are the generators of $su(2)$ as in \eqref{su2}. Clearly, we have
\[
d\hat{e}{^1}=-\lambda^2\, \hat{e}{^2}\wedge\hat{e}{^3}, \quad d\hat{e}{^2}=- \hat{e}{^3}\wedge\hat{e}{^1}, \quad d\hat{e^3}=- \hat{e}{^1}\wedge\hat{e}{^2},
\]
which as $\lambda\rightarrow 0^+$ corresponds to a contraction to the Lie algebra $VI_0$ given in \eqref{lor}. For each $\lambda>0$ we have that the metric $\hat g = \hat{f_1}^2\, (\hat{e}^1)^2+\hat{f_2}^2\, (\hat{e}^2)^2+\hat{f_3}^2\, (\hat{e}^3)^2 +\hat{h}(x)^2\, dx^2$, $\hat{f}_i=\hat{f}_i(x)$ for $i=1,2,3$ is a hyper-K\"ahler metric provided the following system of ODEs holds true
\begin{equation}  \label{evol4f-bis}
\frac{d}{d x}(\hat{f}_1\hat{f}_2)=\hat{f}_3,\quad \frac d{dx}(\hat{f}_1\hat{f}_3)=\hat{f}_2, \quad \frac
d{dx}(\hat{f}_2\hat{f}_3) = \lambda^2 \hat{f}_1.
\end{equation}
Letting $\lambda\rightarrow 0^+$  we obtain  the system \eqref{evol4flor}.  In this sense, the metric given by \eqref{hyplor} is obtained from the metric defined by \eqref{e:tri-axialform} and \eqref{e:su2} using the contraction between the corresponding Lie algebras. The other possible contractions can be treated analogously. }

\section{Hyper symplectic  metrics in dimension 4}

In this section, following the method of the preceding section, we
present explicit {hyper-symplectic (also called
hyper-parak\"ahler)} metrics in dimension four, of signature
{(2,2). For this, we lift the special structure} on the
non-Euclidean Bianchi type groups of class A, discovered in the
preceding section (Proposition~\ref{hitsu2}), to a
{hyper-symplectic} metric on its product with the real line. The
construction extends the correspondence between Bianchi type IX
hyper-K\"ahler metrics and Bianchi type VIII
{hyper-parak\"ahler} structures discovered in \cite{KM}.

First, we recall that an \emph{almost hyper-paracomplex structure} on a $4n$-dimensional manifold $%
M $ is a triple $(J,P_1,P_2)$ of endomorphisms of the tangent bundle of $M$
satisfying the paraquaternionic identities, namely,
\begin{equation*}
J^2=-P_1^2=-P_2^2=-1, \quad JP_1=-P_1J=P_2.
\end{equation*}
If in addition, $J,P_1$ and $P_2$ are \emph{integrable} (that is,
its corresponding Nijenhuis tensor vanishes), the almost hyper-paracomplex
structure $(J,P_1,P_2)$ on $M$ is called \emph{hyper-paracomplex} structure.
(The Nijenhuis tensor of an endomorphism $P$ of the tangent bundle of $M$ is
given by
\begin{equation*}
N_P(X,Y)=[PX,PY]-P[PX,Y]-P[X,PY]+P^2[X,Y],
\end{equation*}
for all vector fields $X, Y$ on $M$.)

An almost hyper-paracomplex manifold $(M,J,P_1,P_2)$ is said to be
\emph{almost hyper-parahermitian} if there exists a {pseudo-Riemannian} metric $g$ satisfying the compatibility conditions
\begin{equation*}
g(J\cdot,J\cdot)=-g(P_1\cdot,P_1\cdot)=-g(P_2\cdot,P_2\cdot)=g(\cdot,\cdot).
\end{equation*}
The compatible metric $g$ is necessarily of neutral
signature $(2n,2n)$ because, at each point of $M$, there is a local
pseudo-orthonormal frame of vector fields given by
\begin{equation}\label{e:ONB}
\{E_1,\ldots,E_n,JE_1,\ldots,JE_n,P_1E_1,\ldots,P_1E_n,P_2E_1,\ldots,P_2E_n%
\}.
\end{equation}
The \emph{fundamental 2-forms}  are the differential $2$-forms on $M$ defined by
\begin{equation}\label{e:neutral 2-forms}
\Omega_1=g(\cdot,J\cdot),\quad \Omega_2=g(\cdot,P_1\cdot), \quad \Omega_3=g(\cdot,P_2\cdot).
\end{equation}
When these forms are closed the almost hyper-parahermitian structure $%
(g,J,P_1,P_2)$ is said to be \emph{hyper-symplectic} \cite{Hit1} or \emph{hyper-parak\"ahler} \cite{IZ}. This
implies (adapting the computations of Atiyah-Hitchin \cite{AH} for
hyper-K\"ahler manifolds) that the structures $J,P_1$ and $P_2$ are
integrable and parallel with respect to the Levi-Civita connection \cite%
{Hit1,DJS}. {In dimension four  an almost hyper-paracomplex
structure is locally equivalent to an oriented neutral conformal
structure (or an $Sp(1,\mathbb{R})$ structure) and the
integrability implies the anti-self-duality of the corresponding
neutral conformal structure \cite{Kam}.} {In particular, a
four dimensional {hyper-symplectic} structure  is equivalent to
an anti-self-dual (ASD) {and} Ricci-flat neutral metric.} For this
reason such structures have been used in string theory
\cite{OV,hul,JR,Bar,Hull,CHO} and integrable systems
\cite{D12,BM,DW}.

Let $G$ be a three dimensional Lie group and let $\{e^1(t),e^2(t),e^3(t)\}$ be a global basis of 1-forms on $G$
for each $t\in I$, where $I\subset \mathbb{R}$ is a connected interval in the real line.
We consider
the almost hyper-parahermitian structure (or $Sp(1,\mathbb{R})$-structure) on
$G\times I$
defined by the following 2-forms
\begin{equation}\label{p4inst}
\begin{aligned}
\Omega_1=& -e^1(t)\wedge e^2(t)+e^3(t)\wedge f(t)dt, \\
\Omega_2=& e^1(t)\wedge e^3(t)-e^2(t)\wedge f(t)dt, \\
\Omega_3=& e^2(t)\wedge e^3(t)+e^1(t)\wedge f(t)dt,
\end{aligned}
\end{equation}
where $f(t)$ is function of $t\in I$ which does not vanish.

{We use the ordered pseudo-orthonormal basis given by
\eqref{e:ONB} to orient \emph{negatively} the manifold $M=G\times
I$. Then the fundamental 2-forms \eqref{p4inst} constitute a basis
of the self-dual (SD) 2-forms and a hyper-symplectic structure of
the form \eqref{p4inst} is equivalent to an anti-self-dual (ASD)
and Ricci-flat neutral metric.}

With the help of Hitchin's theorem \cite{Hit1}, which states that an almost
hyper-parahermitian structure is hyper-symplectic
exactly when the fundamental 2-forms are closed, $d\Omega_s=0$, $s=1,2,3$, it is straightforward to prove similarly to Proposition~\ref{hitsu2} the
following important fact.

\begin{prop}
\label{hitres} The almost hyper-parahermitian structure $(\Omega_1,\Omega_2,%
\Omega_3)$ is
{hyper-symplectic} if and only if the conditions \eqref{4inst1} are
satisfied
and the following evolution equations hold
\begin{equation}  \label{pevol4}
\frac{\partial}{\partial t}e^{12}(t)=f(t)de^3(t),\quad \frac{\partial}{%
\partial t}e^{13}(t)=f(t)de^2(t), \quad \frac{\partial}{\partial t}%
e^{23}(t)=-f(t)de^1(t).
\end{equation}
The hyper-symplectic metric is given by
\begin{equation}  \label{phypkel4}
g=(e^1(t))^2+(e^2(t))^2-(e^3(t))^2-f^2(t)dt^2.
\end{equation}
\end{prop}

As in the previous section, from Lemma~\ref{left-invariant
evolution} it follows that the above Proposition can be applied to
the evolution \eqref{ev} in the case of the Bianchi type A groups
only. For each of the  Bianchi type A groups we shall construct
explicitly the general triaxial hyper-symplectic metric, which is
of the form
\begin{equation}\label{e:tri-axial form}
g=f_1^2(e_1)^2+f_2^2(e_2)^2-f_3^2(e_3)^2-f^2dt^2,
\end{equation}
by solving the corresponding system for the functions
$f_1,f_2,f_3,f$.  We shall see that in each case, we obtain a
system identical to one of the systems encountered in the
hyper-k\"ahler case.

\subsection{Bianchi type $IX$ hyper-symplectic metrics and hyper-K\"ahler Bianchi type $VIII$ hyper-K\"ahler metrics}
Let $G=SU(2)=S^3$ be described by the structure equations \eqref{su2}.
 We evolve the $SU(2)$ structure
according to \eqref{ev}.

Using \eqref{su2}, we reduce the evolution equations
\eqref{pevol4} to the {already considered system
\eqref{evol4f11}.} This establishes a correspondence between
triaxial Bianchi type $IX$ { hyper-symplectic metrics} and
{triaxial} hyper-K\"ahler Bianchi type $VIII$ hyper-K\"ahler metrics.

The general solution is given by \eqref{e:su2g}. Taking $f=f_1f_2f_3$ in %
\eqref{e:su2g} and all $f_i$ different we obtain explicit expression of a
triaxial  {hyper-symplectic} metric \eqref{e:tri-axial form},
{where the forms $e^1,e^2,e^3$ and the functions $f_1,f_2,f_3,f$ are given by \eqref{eulersu2} and \eqref{e:su2g}, respectively.}

A particular solution is obtained by letting $a_1=a_2=0, a_3=\frac{a}{16}$
in \eqref{e:su2g} which gives
\begin{equation*}
f_1=f_2=\frac12(a-t^4)^{\frac14}, \quad f_3=\frac{t^2}2(a-t^4)^{-\frac14},
\quad f=t(a-t^4)^{-\frac14},
{\quad 0<t^4<a.}
\end{equation*}
The resulting {hyper-symplectic } metric is given by
\begin{equation*}
g=\frac14(a-t^4)^{\frac12}\Big(d\theta^2+\sin^2\theta\, d\phi^2\Big)-\frac{t^4%
}{4(a-t^4){^\frac12}} \Big(d\psi+\cos\theta d\phi\Big)^2-\frac{t^2}{%
(a-t^4)^{\frac12}}dt^2.
\end{equation*}

\subsection{Bianchi type $VIII$ hyper-symplectic metrics and hyper-K\"ahler Bianchi type $IX$ hyper-K\"ahler metrics}

Let $G=SU(1,1)$ be defined by the structure equations \eqref{su11}.
 We {evolve} the $SU(1,1)$ structure as
in \eqref{ev}. Using the structure equations \eqref{su11}, the evolution
equations \eqref{pevol4} reduce to the already solved system \eqref{evol4f}. The general solution is of the form \eqref{e:su2} {which has also the expression \eqref{triax}. A substitution of \eqref{triax} and \eqref{su11c} in \eqref{e:tri-axial form} gives the corresponding triaxial {hyper-symplectic} metrics}.

This establishes a correspondence between triaxial Bianchi type
$IX$ hyper-K\"ahler metrics {and {triaxial} Bianchi type $VIII$
hyper-symplectic metrics} discovered in \cite{KM}.

A particular solution to \eqref{evol4f} is given by \eqref{ehfff}, which
results { a  hyper-symplectic} metric in Eguchi-Hanson form \cite%
{DS,T,KM} given by
\begin{multline*}
g=\frac{t^2}4\Big[\Big(d\psi -\cos\theta d\phi\Big)^2+\Big(\sinh\psi
d\theta+\cosh\psi\sin\theta d\phi\Big)^2\Big] \\
-\frac{t^2}4\Big(1-\frac{a}{t^4}\Big)\Big(\cosh\psi
d\theta+\sinh\psi\sin\theta d\phi\Big)^2- \Big(1-\frac{a}{t^4}\Big)%
^{-1}(dt)^2.
\end{multline*}

{Setting $f=-\frac{f_3}t$} one obtains another
hyper-symplectic metric.

\subsection{Bianchi type $II$ hyper-symplectic metrics and hyper-K\"ahler Bianchi type $II$ hyper-K\"ahler metrics}

Consider the two-step nilpotent Heisenberg group $H^3$ defined by the
structure equations \eqref{heis3}. We  evolve the structure as in \eqref{ev}. The structure equations \eqref{heis3}
reduce the evolution equations \eqref{pevol4} to the already solved system %
\eqref{evol4h} {taking $-f_3$ instead of $f_3$. This is  equivalent to consider the two-step nilpotent Heisenberg group $H^3$ defined by the structure equations
$$de_1=de^2=0, \quad de^3=e^{12},\qquad e^1=dx, \quad e^2=dy,\quad e^3=dz+\frac12xdy-\frac12ydx$$
and evolving the structure as in \eqref{ev}.}

The general solution is of the form \eqref{e:H sols for xi}. This
establishes the corresponding form of the general triaxial
hyper-symplectic metric \eqref{e:tri-axial form} where the
functions $f_1,f_2,-f_3,f$ and the 1-forms $e^1,e^2,e^3$ are given
by \eqref{e:H sols for xi} and \eqref{heis33}, respectively.

A particular solution is $f_1=f_2=f=t^{\frac12}, \quad f_3=-t^{-\frac12}$,
 with $t>0$.
This is the hyper-symplectic metric
\begin{equation*}
g=t\Big(-dt^2+dx^2+dy^2\Big)-\frac1t\Big(dz-\frac12xdy+\frac12ydx\Big)^2.
\end{equation*}

\subsection{Bianchi type $VII_0$ hyper-symplectic metrics and hyper-K\"ahler Bianchi type $VI_0$ metrics}

We consider the group $E_2$ of rigid motions of Euclidean 2-plane defined by
the structure equations \eqref{euc}.
We evolve the structure as in \eqref{ev}. Using the structure equations %
\eqref{euc}, the evolution equations \eqref{pevol4} take the form of the
already solved system of ODEs \eqref{evol4flor} with a general solution %
\eqref{e:evol4flor} giving a correspondence with Bianchi $VI_0$
{hyper-K\"ahler metrics.}

When $f_2=f_3^{-1}$,  $ f_1=f$ we have
\begin{equation*}
f=f_1=(a\cos t+b\sin t)^{\frac12}(a\sin t-b\cos t)^{\frac12},\qquad
f_3=f_2^{-1}=(a\sin t-b\cos t)^{\frac12}(a\cos t+b\sin t)^{-\frac12}.
\end{equation*}
Introducing $t_0$ and $r_0$ by letting $r_0=\sqrt{a^2+b^2}$, $\cos t_0=a/\sqrt{a^2+b^2}$ and $\sin t_0=b/\sqrt{a^2+b^2}$, the resulting hyper-symplectic metric can be put in the form
\begin{equation}  \label{e:phyplor}
g=\frac12\,r_0^2\sin 2(t-t_0)\bigl(-dt^2+d\phi^2\bigr)\\
+\cot(t-t_0)(e^2)^2 - \tan(t-t_0)(e^3)^2,
\end{equation}
where $e^2,e^3$ are given by \eqref{euc}.
After an obvious reparametrization,
this metric can be written as
\[
g=\frac12\, r_0^2\sin{2\tau}\big(-d\tau^2+d\phi^2\big)+\cot \tau\big(\sin\phi \,dx
-\cos\phi\, dy\big)^2 -\tan \tau\big(\cos\phi\, dx+\sin\phi\, dy\big)^2.
\]

\subsection{Bianchi type $VI_0$ hyper-symplectic metrics and hyper-K\"ahler Bianchi type $VII_0$ metrics}

Now we consider the group of rigid motions $E(1,1)$ of Lorentzian 2-plane
defined by the structure equations \eqref{lor}. We evolve the structure as
in \eqref{ev}. Using the structure equations \eqref{lor}, the evolution
equations \eqref{pevol4} turn into the solved system of ODEs %
\eqref{evol4feuc} with the general solution given by
\eqref{e:Emot} establishing a correspondence with Bianchi $VI_0$
{hyper-K\"ahler metrics.}

When $f_2=f_3^{-1}, \quad f_1=f$ we have
\begin{equation*}
f=f_1=\frac12(Ae^t+Be^{-t})^{\frac12}(Ae^t-Be^{-t})^{\frac12},\quad
f_3=f_2^{-1}=(Ae^t+Be^{-t})^{-\frac12}(Ae^t-Be^{-t})^{\frac12},
\end{equation*}
and the hyper-symplectic metric is
\begin{equation}  \label{phypeuc}
g=\frac14(A^2e^{2t}-B^2e^{-2t})\Big(-dt^2+d\phi^2+
\frac4{(Ae^t-Be^{-t})^2}(e^2)^2 -\frac4{(Ae^t+Be^{-t})^2}(e^3)^2\Big),
\end{equation}
where $e^2,e^3$ are given by \eqref{lor}
{, and $t>\frac{1}{2}\log\left|\frac {B}{A}\right|$
since $A^2 e^{2t}-B^2 e^{-2t}>0$}.  In particular, letting $A=B$ in \eqref{phypeuc} we obtain
\begin{equation*}
g=\frac{A^2}2\sinh{2t}\Big(-dt^2+d\phi^2\Big)+\coth t\Big( \sinh\phi \,
dx+\cosh\phi\, dy\Big)^2 -\tanh t\Big( \cosh\phi \, dx+\sinh\phi\, dy\Big)^2.
\end{equation*}

\section{Conclusions}
 {Several explicit  cohomogeneity one hyper-K\"ahler (local) metrics of Riemannian and neutral signature were found on four dimensional manifolds $M=G\times I$, $I=(a,b)$.  The group $G$ was assumed to be a three dimensional Bianchi type group, which acts on  $M$ by left translations on the first factor of $M$.  By considering a time dependent evolution of a fixed  left invariant basis of one forms on $G$ we defined an almost hyper-Hermitian structure on $M$. The action of $G$ extends to a trivial  action on the fundamental 2-forms of the hyper-Hermitian structure.  The  hyper-Hermitian structure is hyper-K\"ahler if the defined fundamental 2-forms are closed.  It was found that the latter condition can be fulfilled only in the case of Bianchi type A groups. Restricting the evolution to one of  diagonal type we obtained a system of ODEs whose solutions give a hyper-K\"ahler  metric. For each of the Bianchi type A groups we found the explicit solution of the system of ODEs.  The hyper-K\"ahler metrics of neutral signature, called hyper-symplectic metrics, were found to be in correspondence with the Riemannian  hyper-K\"ahler metrics based on the type of the system of ODEs. Thus, we give new examples of hyper-symplectic metrics and extend the correspondence between  classes of hyper-symplectic metrics and hyper-K\"ahler metrics given in \cite{KM}  in the Bianchi VIII and IX cases to all Bianchi type A cohomogeneity one (with trivial action on the fundamental 2-forms) metrics.
}

\end{document}